\documentclass[envcountsect]{llncs}
\usepackage[utf8]{inputenc}
\usepackage{amssymb,amsmath,mathrsfs}
\usepackage[english]{babel}
\usepackage[unicode,unicode=true,bookmarks=false]{hyperref}
\usepackage[shortlabels]{enumitem}

\usepackage{environ}

\newcommand{\repeattheorem}[1]{%
  \begingroup
  \renewcommand{\thetheorem}{\ref{#1}}%
  \expandafter\expandafter\expandafter\theorem
  \csname reptheorem@#1\endcsname
  \endtheorem
  \endgroup
}

\NewEnviron{reptheorem}[1]{%
  \global\expandafter\xdef\csname reptheorem@#1\endcsname{%
    \unexpanded\expandafter{\BODY}%
  }%
  \expandafter\theorem\BODY\unskip\label{#1}\endtheorem
}

\newcommand{\sco}{,\ldots,}

\newcommand{\RR}{\mathbb{R}}
\newcommand{\NN}{\mathbb{N}}
\newcommand{\ZZ}{\mathbb{Z}}
\newcommand{\QQ}{\mathbb{Q}}
\newcommand{\CC}{\mathbb{C}}

\newcommand{\PrA}{\mathop{\mathbf{PrA}}\nolimits}
\newcommand{\PA}{\mathop{\mathbf{PA}}\nolimits}
\newcommand{\ZF}{\mathop{\mathbf{ZF}}\nolimits}

\makeatletter

\renewcommand{\epsilon}{\varepsilon}
\renewcommand{\phi}{\varphi}
\newcommand{\sref}[2]{\hyperref[#2]{#1 \ref*{#2}}}
\newcommand{\dref}[2]{\hyperref[#2]{ #1 }}

\newcommand{\rk}{\mathsf{rk}}

\newcommand{\Kc}{\mathcal{K}}
\newcommand{\Lc}{\mathcal{L}}

\newcommand{\eqdef}{\stackrel{\mbox{\tiny\rm def}}{=}}\renewcommand{\to}{\rightarrow}
\newcommand{\ra}{\rightarrow}

\newcommand{\Lra}{\Leftrightarrow}
\newcommand{\lra}{\leftrightarrow}

\spnewtheorem{hyp}{Conjecture}[section]{\bfseries}{\itshape}
\spnewtheorem{ex}{Example}{\bfseries}{\itshape}

\newcommand{\xv}{\overline{x}}

\newcommand{\av}{\overline{a}}

\newcommand{\vv}{\overline{v}}
\newcommand{\uv}{\overline{u}}
\newcommand{\cv}{\overline{c}}

\begin{document}
	\author{Fedor Pakhomov$^{12}$ and Alexander Zapryagaev$^3$}
	\title{Multi-Dimensional Interpretations of Presburger Arithmetic in Itself}
	\institute{${}^1$Institute of Mathematics of the Czech Academy of Sciences, Žitná 25, 115 67, Praha 1, Czech Republic\\ ${}^2$Steklov Mathematical Institute of Russian Academy of Sciences, 8, Gubkina Str., Moscow, 119991, Russian Federation\\ ${}^3$National Research University Higher School of Economics, 6, Usacheva Str., Moscow, 119048, Russian Federation}
	
	\maketitle
	
	\begin{abstract}
          Presburger Arithmetic is the true theory of natural numbers with addition. We study interpretations of Presburger Arithmetic in itself. The main result of this paper is that all self-interpretations are definably isomorphic to the trivial one. Here we consider interpretations that might be multi-dimensional.  We note that this resolves a conjecture by A.~Visser. In order to prove the result we show that all linear orderings that are interpretable in $(\NN,+)$ are scattered orderings with the finite Hausdorff rank and that the ranks are bounded in the terms of the dimensions of the respective interpretations.
	\end{abstract}
	
	\section{Introduction}
	
	Presburger Arithmetic $\PrA$ is the true theory of natural numbers with addition. Unlike Peano Arithmetic $\PA$, it is complete, decidable and admits quantifier elimination in an extension of its language\cite{presburger}.
	
	The method of interpretations is a standard tool in model theory and in the study of decidability of first-order theories \cite{tarskimostowski,hodges}. An interpretation of a theory $\mathbf{T}$ in a theory $\mathbf{U}$ is essentially a uniform first-order definition of models of $\mathbf{T}$ in models of $\mathbf{U}$ (see details in Section~2). In the paper we study certain questions about interpretability for Presburger Arithmetic that were well-studied in the case of stronger theories like Peano Arithmetic $\PA$. Although from technical point of view the study of interpretability for Presburger Arithmetic uses completely different methods than the study of interpretability for $\PA$ (see, for example, \cite{visser}), we show that from interpretation-theoretic point of view, $\PrA$ has certain similarities to strong theories that prove all the instances of mathematical induction in their own language, i.e. $\PA$, Zermelo-Fraenkel set theory $\ZF,$ etc.
	
	A \emph{reflexive} arithmetical theory (\cite[p.\,13]{visser}) is a theory that can prove the consistency of all its finitely axiomatizable subtheories. Peano Arithmetic $\PA$ and Zermelo-Fraenkel set theory $\ZF$ are among well-known reflexive theories. In fact, all sequential theories (very general class of theories similar to $\PA,$ see \cite[III.1(b)]{hajekpudlak}) that prove all instances of the induction scheme in their language are reflexive. For sequential theories reflexivity implies that the theory cannot be interpreted in any of its finite subtheories.  A.~Visser has conjectured that this purely interpretational-theoretic property holds for $\PrA$ as well. Note that $\PrA$ satisfies full induction scheme in its own language but cannot formalize the statements about consistency of formal theories.
	
	Unlike sequential theories, Presburger Arithmetic cannot encode tuples of natural numbers by single natural numbers. And thus, for interpretations in Presburger Arithmetic it is important whether individual objects are interpreted by individual objects (one-dimensional interpretations) or by tuples of objects of some fixed length $m$ ($m$-dimensional interpretations).
	
	J.~Zoethout \cite{jetze} considered the case of one-dimensional interpretations and proved that if any one-dimensional interpretation of $\PrA$ in $(\NN,+)$ gives a model definably isomorphic to $(\mathbb{N},+)$, then Visser's conjecture holds for one-dimensional interpretations, i.e. there are no one-dimensional interpretations of $\PrA$ in its finite subtheories. Moreover, he proved that any interpretation of $\PrA$ in $(\NN,+)$ is isomorphic to $(\NN,+)$; however, he hadn't proved that the isomorphism is definable. We improve the latter result and establish the definability of the isomorphism.
	
	\begin{reptheorem}{1a}
		The following holds for any  model $\mathfrak{A}$ of $\PrA$  that is one-di\-men\-sionally interpreted in the model $(\NN,+)$:   
		\begin{enumerate}[(a)] 
		\item $\mathfrak{A}$ is isomorphic to $(\NN,+)$,
		\item the isomorphism is definable in $(\NN,+)$.
		\end{enumerate}
	\end{reptheorem}
	
	Then, by a more sophisticated technique, we establish Visser's conjecture for multi-dimensional interpretations.
	
	\begin{reptheorem}{1b}
		The following holds for any model $\mathfrak{A}$ of $\PrA$  that is interpreted in $(\NN,+)$: 
		\begin{enumerate}[(a)] 
		\item $\mathfrak{A}$ is isomorphic to $(\NN,+)$,
		\item the isomorphism is definable in $(\NN,+)$.
		\end{enumerate}
	\end{reptheorem}
	
	In the present paper we obtain both \sref{Theorem}{1a} (a) and \sref{Theorem}{1b} (a) as a corollary of a single fact about linear orderings interpretable in $(\NN,+)$. Recall that any non-standard model of Presburger arithmetic has the order type of the form $\mathbb{N}+\mathbb{Z}\cdot A$, where $A$ is a dense linear ordering. In particular, it means that the order types of non-standard models of $\PrA$ are never scattered (a linear ordering is called \emph{scattered} if it contains no dense suborderings). We show that any linear ordering that is interpretable in $(\mathbb{N},+)$ is scattered.
	
	In fact, we establish an even sharper result and estimate the ranks of the interpreted orderings. The standard notion of rank of a scattered linear ordering is the  Cantor-Bendixson rank that goes back to Hausdorff \cite{hausdorff}. However, in our case a more precise estimation is obtained using a slightly different notion of $\mathit{VD}_*$-rank from \cite{krs}.
	
	\begin{reptheorem}{ordering}
		Suppose a linear ordering $(L,\prec)$ is $m$-dimensionally interpretable in $(\NN,+)$. Then $(L,\prec)$ is scattered and has $\mathit{VD}_*$-rank at most $m$.
	\end{reptheorem}

	In order to prove \sref{Theorem}{1a} (b), we show that the (unique) isomorphism of the interpreted model $\mathfrak{A}$ and $(\NN,+)$ is in fact definable in $(\NN,+)$. This isomorphism is trivially definable using counting quantifiers, while the theorem that in Presburger Arithmetic first-order formulas with counting quantifiers have the same expressive power as ordinary first-order formulas is due to H.~Apelt \cite{apelt} and N.~Schweikardt \cite{schweikardt}.
	
	The proof of \sref{Theorem}{1b} relies on a theory of cardinality functions $p\mapsto |A_p|$ for definable families of finite sets $\langle A_p\subseteq \NN^m\mid p\in P\subseteq \NN^n\rangle$. 
	
	We note that the present work essentially is an expanded version of the paper \cite{zp}. Results of \sref{Theorem}{1a}(a,b), \sref{Theorem}{1b}(a), \sref{Theorem}{ordering}, \sref{Theorem}{bash}, and \sref{Corollary}{card_funct}  were already present in \cite{zp}. \sref{Theorem}{1b}(b) is new.
	
	The work is organized in the following way. Section 2 introduces Presburger Arithmetic and interpretations. In Section 3, we define notion of dimension for Presburger-definable sets and prove \sref{Theorem}{ordering}. In Section 4 we prove \sref{Theorem}{1a}. In Section 5 we prove \sref{Theorem}{1b}.
	
	\section{Preliminaries}
	
	\subsection{Presburger Arithmetic}
	
	In this section we give some general results about Presburger Arithmetic and definable sets in $(\mathbb{N},+)$.  In this paper the set of all natural numbers $\NN$ includes zero.
		
	\begin{definition}
		{\em Presburger Arithmetic} $(\PrA)$ is the elementary theory of the model $(\NN,+)$ of natural numbers with addition.
	\end{definition}
	
	It is easy to define the constants $0,1$, relation $\le$ and modulo comparison relations $\equiv_n$, for all $n\ge 1$, in the model $(\NN,+)$. In the language extended by these constants and predicates, Presburger arithmetic admits quantifier elimination \cite{presburger}. Furthermore, $\PrA$ is decidable.
	
	
	$\PrA$ has non-standard models. Unlike $\PA$, however, where it is impossible to produce an explicit non-standard model by defining some recursive addition and multiplication (Tennenbaum's Theorem \cite{tennenbaum}), examples of non-standard models of $\PrA$ can be given explicitly (see \cite{smorynski}). By a usual argument one can show that any non-standard model of $\PrA$ has the order type $\NN+\ZZ\cdot L$, where $L$ is a dense linear ordering without endpoints. In particular, any countable model of $\PrA$ has the order type of either $\NN$ or $\NN+\ZZ\cdot\QQ.$
	
		
	
	\begin{definition}
		For vectors $\overline{c},\overline{p_1},\ldots,\overline{p_n}\in\ZZ^m$ we call the set $\{\overline{c}+\sum k_i\overline{p_i}\mid k_i\in\NN\}\subseteq \ZZ^m$ a \emph{lattice} (or a \emph{linear set}) generated by $\{\overline{p_i}\}$ from $\overline{c}$. If $\{\overline{p_i}\}$ are linearly independent, we call this set a \emph{fundamental lattice}.
	\end{definition}
	
	According to \cite{ginsburg}, definable subsets of $\NN^m$ are exactly the unions of a finite number of (possibly intersecting, possibly non-fundamental) lattices (such unions are also called \emph{semilinear sets} in literature). Ito has shown in \cite{ir} that any set in $\NN^m$ which is a union of a finite number of (possibly intersecting, possibly non-fundamental) lattices (a semilinear set) can be expressed as a union of a finite number of disjoint fundamental lattices. Hence,
	
	\begin{theorem} \label{fund}
		All subsets of $\NN^k$ definable in $(\mathbb{N},+)$ are exactly the subsets of $\NN^k$ that are disjoint unions of finitely many fundamental lattices.
	\end{theorem}
	
	\begin{definition}
	 For a fundamental lattice $J$ generated by $\overline{v}_1,\ldots,\overline{v}_n$ from $\overline{c}$ we call a function $f\colon J\to \NN$ \emph{linear} if it is of the form $f(\overline{c}+x_1\overline{v}_1+\ldots+x_n\overline{v}_n)=a_0+a_1x_1+\ldots+a_nx_n$ for some $a_0,\ldots,a_n\in\NN$.
	 
	 For an $(\NN,+)$-definable set $A$ we call a function $f\colon A\to \NN$ \emph{piecewise linear} if there is a decomposition of $A$ into disjoint fundamental lattices $J_1,\ldots,J_n$ such that the restriction of $f$ on any $J_i$ is linear\footnote{In our work, we use the word `piecewise' only in the sense of the current definition.}. 
	\end{definition}

	\begin{theorem}\label{jetz}
		Functions $f\colon\NN^n\ra\NN$ definable in $(\NN,+)$ are exactly piecewise linear functions.
	\end{theorem}

	\begin{proof}
		The definability of all piecewise linear functions in Presburger Arithmetic is obvious. A function $f\colon \NN^n\to \NN$ is definable if and only if its graph
		
		\begin{center}
			$G=\{(a_1,\ldots,a_n,f(a_1,\ldots,a_n))\mid (a_1,\ldots,a_n)\in \NN^n\}$
		\end{center}
		\noindent
		is definable. According to \sref{Theorem}{fund}, $G$ is a finite union of fundamental lattices $J_1\sqcup\ldots \sqcup J_k$. For $1\le i\le k$ we denote by $J_i'$ the projections of $J_i$ along the last coordinate, $J_i'= \{(a_1,\ldots,a_n)\mid\exists a_{n+1} ( (a_0,a_1,\ldots,a_n,a_{n+1})\in J_i)\}$. Clearly, all $J_i'$ are fundamental lattices. Furthermore, the restriction of the function $f$ on each of $J_i'$ is linear.
	\end{proof}
	
    \subsection{Interpretations}
    
	We define the notion of a multi-dimensional first-order non-parametric interpretation, following \cite{tarskimostowski}.
	
	\begin{definition}
		An \emph{$m$-dimensional interpretation} $\iota$ of some first-order language $\Kc$  in a model $\mathfrak{A}$ consists of first-order formulas of language of $\mathfrak{A}$:
		
		\begin{enumerate}
			\item $D_{\iota}(\overline{y})$ defining the set $\mathbf{D}_{\iota}\subseteq \mathfrak{A}^m$ (domain of interpreted model);			
			\item $P_{\iota}(\overline{x}_1,\ldots,\overline{x}_n)$, for predicate symbols $P(x_1,\ldots,x_n)$ of $\Kc$ including equality;			
			\item $f_\iota(\overline{x}_1,\ldots,\overline{x}_n,\overline{y})$, for functional symbols $f(x_1,\ldots,x_n)$ of $\Kc$.
		\end{enumerate}
	\end{definition}
	
	Here all vectors of variables $\overline{x}$ are of length $m$, and $f_\iota$'s should define graphs of some functions (modulo interpretation of equality).
	
	Naturally, $\iota$ and $\mathfrak{A}$ give a model $\mathfrak{B}$ of the language $\Kc$ on the domain $\mathbf{D}_{\iota}/{\sim_{\iota}}$, where equivalence relation $\sim_{\iota}$ is given by $=_{\iota}(\overline{x}_1,\overline{x}_2)$. We will call $\mathfrak{B}$ the \emph{internal model}.	
	
	If $\mathfrak{B}\models \mathbf{T}$, then $\iota$ is an \emph{interpretation of the theory} $\mathbf{T}$ in $\mathfrak{A}$. If for a first-order theory $\mathbf{U}$ an interpretation $\iota$ is an interpretation of $\mathbf{T}$, for any $\mathfrak{A}\models \mathbf{U}$, then $\iota$ is an interpretation of $\mathbf{T}$ in $\mathbf{U}$.
	
	Interpretations are a very natural concept, appearing in mathematics when, for example, Euclidean geometry is interpreted in the theory of real numbers $\RR$ (two-dimensionally, by defining points as pairs of real numbers) in analytic geometry, or the field $\CC$ of complex numbers is two-dimensionally interpreted in $\RR$ by defining $a+bi\lra(a,b)$ and addition and multiplication are declared by definition. We note that in $(\NN,+)$ itself, the field $(\ZZ,+)$ can be interpreted. This is achieved by mapping the negative numbers to odd, positive to even and $0$ to $0$ and defining the addition case-by-case (through non-negative subtraction, which is definable).
	
	We will be interested in interpretations of theories in the standard model of Presburger Arithmetic, that is, in $(\mathbb{N},+)$.
	
	\begin{definition}
		An $m$-dimensional interpretation $\iota$ in a model $\mathfrak{A}$ \emph{has absolute equality} if the symbol $=\in\Kc$ is interpreted as the coincidence of two $m$-tuples.
	\end{definition}
	
	\begin{definition}
		An interpretation $\iota,\kappa$ in a model $\mathfrak{A}$ are \emph{definably isomorphic}, if there is a first-order formula $F(\overline{x},\overline{y})$. of the language of $\mathfrak{A}$ defining an isomorphism between the respective internal models. 
	\end{definition}

	The following theorem is a version of \cite[Lemma 3.2.2]{jetze}, extended to multi-dimensional interpretations. It shows that it suffices to consider only the interpretations with absolute equality.
	
	\begin{theorem}\label{4r}
	Suppose $\iota$ is an  interpretation of  some theory $\mathbf{U}$ in $(\NN,+)$. Then there is an interpretation $\kappa$ of $\mathbf{U}$ in $(\NN,+)$ with absolute equality which is definably isomorphic to $\iota$.
	\end{theorem}
	\begin{proof}		Indeed, there is a definable in $(\NN,+)$ well-ordering $\prec$ of $\NN^m$:
		$$(a_0,\ldots,a_{m-1})\prec (b_0,\ldots,b_{m-1})\stackrel{\mbox{\footnotesize \textrm{def}}}{\iff} \exists  i< m (\forall j< i\;(a_j=b_j)\land a_i<b_i).$$
		Now we could define $\kappa$ by taking the definition of $+$ from $\iota$, interpreting the equality trivially, and declaring the domain of $\kappa$ to be the part of the domain of $\iota$ that contains exactly the $\prec$-least elements of equivalence classes with respect to $\iota$-interpretation of equality. It is easy to see that this $\kappa$ is definably isomorphic to $\iota$.
	\end{proof}
	
	\section{Ranks of Interpreted Orderings}
	
	\subsection{Presburger Dimension}
    Henceforth, we will talk only about definability in the model $(\NN,+)$. By a definable set we always mean a set $A\subseteq\NN^n$ definable in $(\NN,+)$, and a definable function $f\colon A\to B$ will always be understood as a function between definable sets $A$ and $B$ that is definable in $(\NN,+)$ itself.
	
	\begin{definition}
		We say that a natural number $k\ge 1$ is the \emph{dimension} $\dim(A)$ of an \emph{infinite} definable set $A\subseteq\NN^m$ if there is a definable bijection between $A$ and $\NN^k.$
	\end{definition}
	
	The following theorem shows that the definition above uniquely defines the dimension for each infinite definable set.
	
	\begin{theorem}\label{Existe}
		Suppose $A\subseteq \NN^n$ is an infinite definable set. Then there is a unique $m\in\NN$ such that there is a Presburger-definable bijection between $A$ and $\NN^m,$ $1\le m\le k.$
	\end{theorem}
	\begin{proof}
		First we show that there is some $m$ possessing the property. According to \sref{Theorem}{fund}, all sets definable in $(\NN,+)$ are disjoint unions of fundamental lattices $J_1,\ldots,J_n$ of the dimensions $k_1,\ldots,k_n$, respectively (the dimension of a fundamental lattice is the number of generating vectors). It is easy to see that for each $J_i$ there is a linear bijection with $\NN^{k_i}$, which is obviously definable. Let us put $m$ to be the maximum of $k_i$'s. 
		
		Now we notice that for each sequence $r_1,\ldots,r_m\in\NN$ and $u=\max(r_1,\ldots,r_m)$, $u\ge 1$, we are able to split $\NN^u$ into a disjoint union of definable sets $B_1,\ldots,B_m$, for which we have definable bijections with $\NN^{r_1},\ldots,\NN^{r_m}$, respectively. This is proved by induction on $m$.
		
		Let us show that there is no other $m$ with this property. Assume the contrary. Then, for some $m_1>m_2$, there is a definable bijection $f\colon\NN^{m_1}\ra\NN^{m_2}$. Let us consider a sequence of expanding cubes $$I_s^{m_1}\eqdef\{(x_1\sco x_k)\mid 0\le x_1\sco x_n\le s\}.$$ We define function $g\colon \NN\to\NN$ to be the function which maps a natural number $x$ to the least $y$ such that $f(I_x^{m_1})\subseteq I_y^{l_2}$. Clearly, $g$ is a definable function. Then there should be some linear function $h\colon \NN\to\NN$ such that $g(x)\le h(x)$, for all $x\in\NN$. But since for each $x\in\NN$ and $y<x^{m_1/m_2}$ the cube $I_x^{m_1}$ contains more points than the cube  $I_y^{m_2},$ from the definition of $g$ we see that $g(x)\ge x^{m_1/m_2}$. This contradicts the linearity of the function $h$.
	\end{proof}
	
	As far as we know, this definition of dimension for Presburger definable sets was first introduced in \cite{cluckers} and restated in \cite{zp}. It can be seen that the dimension of a set $A\subseteq\NN^n$ is equal to the maximal $m$ such that there exists an $m$-dimensional fundamental lattice which is a subset of $A$.
	
	\begin{definition}
	    For a set $A\subseteq \NN^{n+m}$ and $a\in\NN^n$ we define the section $$A\upharpoonright a=\{b\in \NN^m\mid a\frown b\in A\},$$ where $a\frown b$ is the concatenation of the tuples $a$ and $b$.
	\end{definition}
	
	\begin{definition}
	    For a definable set $P\subseteq \NN^n$ a family of sets $\langle A_p\subseteq \NN^m \mid p\in P\rangle$ is called \emph{definable} if  there is a definable set $A\subseteq P\times \NN^m$ such that $A_p=A\upharpoonright p$, for any $p\in P$.
	\end{definition}

	
	\begin{lemma}\label{MS}
		Suppose $\langle A_p\subseteq \NN^n\mid p\in P\rangle$ is a definable family of sets, and the set $P'\subset P$ (possibly undefinable) is such that for $p\in P'$ the sets $A_p$ are $n$-dimensional and pairwise disjoint. Then $P'$ is finite.
	\end{lemma}
	\begin{proof}
        Let us consider the set $A=\{p\frown a\mid p\in P\mbox{ and }a\in A_p\}$. By \sref{Theorem}{fund}, the set $A$ is a disjoint union of finitely many fundamental lattices  $J_i\subseteq \NN^{k+n}.$ It is easy to see that if some set $A_p$ is $n$-dimensional, then for some $i$ the section $J_i\upharpoonright p=\{a\mid p\frown a\in J_i\}$ is an $n$-dimensional set. Thus it is enough to show that for each $J_i$ there are only finitely many $p\in P'$ for which the section $J_i\upharpoonright p$ is an $n$-dimensional set.
	    
	    Let us now assume for a contradiction that for some $J_i$ there are infinitely many $p\in P'$ for which $J_i\upharpoonright p$ are $n$-dimensional sets. Let us consider some $p\in P'$ such that the section $J_i\upharpoonright p$ is an $n$-dimensional set. Then there exists an $n$-dimensional fundamental lattice $K\subseteq J_i\upharpoonright p$. Suppose the generating vectors of $K$ are $\vv_1,\ldots,\vv_n\in \NN^n$ and initial vector of $K$ is $\uv\in \NN^n$. It is easy to see that each vector $\vv_j$ is a non-negative linear combination of generating vectors of $J_i$, since otherwise for large enough $h\in \mathbb{N}$ we would have $\cv+h\vv_j\not\in J_i$. Now notice that for any  $p\in P$ and $\av \in J_i\upharpoonright p$ the $n$-dimensional lattice with generating vectors $\vv_1,\ldots,\vv_n$ and initial vector $\av$ is a subset of $J_i\upharpoonright p$.
	    
	    Thus infinitely many of the sets $A_p$, for $p\in P'$, contain some shifts of the same $n$-dimensional fundamental lattice $K$. It is easy to see that the latter contradicts the assumption that all the sets $A_p$, for $p\in P'$, are disjoint.
	\end{proof}
	
	\subsection{Ranks of Linear Orderings}
	
	\begin{definition}
	 A linear ordering $(L,\prec)$ is called \emph{scattered} (\cite[pp.\,32--33]{rosenstein}) if it does not have an infinite dense subordering.
	\end{definition}
	
	\begin{definition}
		Let $(L,\prec)$ be a linear ordering. We define a family of equivalence relations $\simeq_{\alpha}$, for ordinals $\alpha\in\mathbf{Ord}$ by transfinite recursion:
		
		\begin{itemize}
			\item $\simeq_0$ is just equality;
			\item $\simeq_{\lambda}=\bigcup\limits_{\beta<\lambda}\simeq_{\alpha}$, for limit ordinals $\lambda$;
			\item $a\simeq_{\alpha+1}b \stackrel{\mbox{\footnotesize $\mathrm{def}$}}{\iff} |\{c\in L\mid (a\prec c\prec b)\mbox{ or }(b\prec c \prec a)\}/{\simeq_{\alpha}}|<\aleph_0$.
		\end{itemize}
	
		Now we define  $\mathit{VD}_*$-\emph{rank}\footnote{$\mathit{VD}$ stand for {\em very discrete}; see \cite[p.\,84-89]{rosenstein}.} $\rk(L,\prec)\in \mathbf{Ord}\cup \{\infty\}$ of the ordering $(L,\prec)$. The $\mathit{VD}_*$-rank $\rk(L,\prec)$ is the least $\alpha$ such that $L/{\simeq_{\alpha}}$ is finite. If, furthermore, for all $\alpha\in \mathbf{Ord}$ the factor-set $L/{\simeq_{\alpha}}$ is infinite, we put $\rk(L,\prec)=\infty$.
		
		By definition we put $\alpha<\infty$, for all $\alpha\in \mathbf{Ord}$.
	\end{definition}
	
	The definition given above corresponds to the procedure of \emph{condensation} that glues the points at finite distance from each other. 
	The $\mathit{VD}_*$-rank is now the minimal number of iterated condensations required to reach some finite ordering.
	
	\begin{proposition}\label{3.1}
		Linear orderings $(L,\prec)$ such that $\rk(L,\prec)<\infty$ are exactly the scattered linear orderings.
	\end{proposition}
	\begin{proof}
		$(\Rightarrow)$ Let $(L,\prec)$ be not scattered. This means there is a dense subordering $(S,\prec)$ in $L$ with the induced order relation. However, after a single condensation operation, any two points of $S$ remain separate as there is an infinite number of points even from $S$ between them. This means that the condensed ordering still contains $S$ as subordering. By transfinite induction, this holds now for all ordinal-numbered iterations. Hence, $L$ cannot have a $\mathit{VD}_*$-rank $<\infty$.
		
		$(\Leftarrow)$ Let $\rk(L,\prec)=\infty$. We have to prove that there is an embedded dense ordering in $L$. Consider the equivalence relation on the points of $L$: $x\sim y\Lra\text{``$x$ and $y$ have been identified on some step of condensations"}$. As the rank does not equal to any ordinal, the number of equivalence classes is infinite. Picking a representative from each, we obtain the required dense subordering: as no two points are identified, there is always an infinite number of points between them. Indeed, were any two groups at a finite distance from each other in the induced ordering, they would have been joined into a single group at some step.
	\end{proof}

	The orderings with the $\mathit{VD}_*$-rank equal to $0$ are exactly finite orderings, and the orderings with $\mathit{VD}_*$-rank $\le 1$ are exactly the ordered sums of finitely many copies of $\NN$, $-\NN$ and $1$ (one-element linear ordering).
	
	\begin{remark}
		Each scattered linear ordering of $\mathit{VD}_*$-rank 1 is $1$-dimensionally interpretable in $(\NN,+)$. There are scattered linear orderings of $\mathit{VD}_*$-rank 2 that are not interpretable in $(\NN,+)$.
	\end{remark}
	\begin{proof}
		The interpretability of linear orderings with rank $0$ and rank $1$ follows from the description above.
		
		Since there are uncountably many non-isomorphic scattered linear orderings of $\mathit{VD}_*$-rank 2 and only countably many linear orderings interpretable in $(\NN,+)$, there is some scattered linear ordering of $\mathit{VD}_*$-rank 2 that is not interpretable in $(\NN,+)$.
	\end{proof}
	
	Now we prove the rank condition.

	\repeattheorem{ordering}
	\begin{proof}
		Since any ordering with a finite $\mathit{VD}_*$ rank is scattered, it is enough that $\rk(L,\prec)\le m$. We prove the theorem by induction on $m\ge 1$.

		Assume for a contradiction that there is an $m$-dimensionally interpretable ordering $(L,\prec)$ with $\rk(L,\prec)>m$. By the definition of $\mathit{VD}_*$-rank, there are infinitely many distinct $\simeq_m$-equivalence classes in $L$. Hence, either there is an infinite ascending $a_0\prec a_1\prec\ldots$ or descending $a_0\succ a_1\succ\ldots$ chain of elements of $L$ such that $a_i\not\simeq_m a_{i+1}$, for each $i$. Let $L_i$ be the intervals $(a_i,a_{i+1})$ in the order $\prec$, if we had an ascending chain, or the intervals $(a_{i+1},a_{i})$ in the order $\prec$, if we had a descending chain. Since $a_i\not\simeq_m a_{i+1}$, the set $L_i/{\simeq_{m-1}}$ is infinite and $\rk(L_i,\prec)>m-1$.

        Clearly, all the intervals $L_i$ are definable. Let us show that $\dim(L_i)\ge m$, for each $i$. If $m=1$ then it follows from the fact that $L_i$ is infinite. If $m>1$ then we assume for a contradiction that $\dim(L_i)<m$. Also notice that in this case $(L_i,\prec)$ would be $(m-1)$-dimensionally interpretable in $(N,+)$, which contradicts the induction hypothesis and the fact that $\rk(L_i,\prec)>m-1$. Since $L_i\subseteq \NN^m$, we conclude that $\dim(L_i)=m$, for all $i$.

		Now consider the definable family of sets $\{(a,b)\mid a,b\in L^2\}$. We see that all $L_i$ are in this family. Thus we have infinitely many disjoint sets of the dimension $m$ in the family and hence there is a contradiction with \sref{Lemma}{MS}.
		
		We have proved that if $(L,\prec)$ is $m$-dimensionally interpretable in $(\NN,+)$, then its $\mathit{VD}_*$-rank is at most $m$. Hence, by \sref{Proposition}{3.1}, $L$ is scattered.
	\end{proof}
	
	\section{Visser's Conjecture in One-Dimensional Case}
	
	Let us now consider the extension of the first-order predicate language with an additional quantifier $\exists^{=y}x,$ called a {\em counting quantifier} (notion introduced in \cite{barrington}). The syntax is as follows: if $f(\xv,z)$ is an $\Lc$-formula with the free variables $\xv,z,$ then $F=\exists^{=y}z\:G(\xv,z)$ is also a formula with the free variables $\xv,y.$
	
	We extend the standard assignment of truth values to first-order formulas in the model $(\mathbb{N},+)$ to formulas with counting quantifiers. For a formula $F(\xv,y)$ of the form $\exists^{=y}z\:G(\xv,z)$, a vector of natural numbers $\av$, and a natural number $n$ we say that $F(\av,n)$ is true if and only if there are exactly $n$ distinct natural numbers $b$ such that $G(\av,b)$ is true. H.~Apelt \cite{apelt} and N.~Schweikardt \cite{schweikardt} have established that such an extension does not change the expressive power of $\PrA:$
	
	\begin{theorem}(\cite[Corollary 5.10]{schweikardt})\label{unti}
		Every formula $F(\overline{x})$ in the language of Presburger arithmetic with counting quantifiers is equivalent in $(\NN,+)$ to a quan\-ti\-fier-free formula.
	\end{theorem}

		\repeattheorem{1a}

	\begin{proof}
	    From \sref{Theorem}{4r} it follows that it is enough to consider the case when the interpretation that gives us $\mathfrak{A}$ has absolute equality.
	
		Let us denote the relation given by the $\PrA$ definition of $<$ within $\mathfrak{A}$ by $<^{\mathfrak{A}}$. Clearly, $<^{\mathfrak{A}}$ is definable in $(\NN,+)$. Hence, by \sref{Theorem}{ordering}, the order type of $\mathfrak{A}$ is scattered. But since any non-standard model of $\PrA$ is not scattered, the model $\mathfrak{A}$ is isomorphic to $(\NN,+)$. 
		
		It is easy to see that the isomorphism $f$ from $\mathfrak{A}$ to $(\NN,+)$ is the function $f\colon x\mapsto |\{y\in\mathbb{N}\mid y<^{\mathfrak{A}}x\}|$. Now we use a counting quantifier to express the function:
		
		\begin{center}
		$$f(a)=b \iff (\mathbb{N},+)\models \exists^{=b}z \;(z<^{\mathfrak{A}}a).$$
		\end{center}
		
		Now apply \sref{Theorem}{unti} and see that $f$ is definable in $(\mathbb{N},+)$.
	\end{proof}
	
	This implies Visser's Conjecture for one-dimensional interpretations.
    
	\begin{theorem}\label{visser_1}
	    Theory $\PrA$ is not one-dimensionally interpretable in any of its finitely axiomatizable subtheories.
	\end{theorem}
	\begin{proof}
	    Assume $\iota$ is an one-dimensional interpretation of $\PrA$ in some finitely axiomatizable subtheory $\mathbf{T}$ of $\PrA$. In the standard model $(\NN,+)$ the interpretation $\iota$ gives us a model $\mathfrak{A}$ for which there is a definable isomorphism $f$ with $(\NN,+)$. Now let us consider theory $\mathbf{T}'$ that consists of $\mathbf{T}$ and the statement that the definition of $f$ establishes an isomorphism between (internal) natural numbers and the structure given by $\iota$. Clearly, $\mathbf{T}'$ is finitely axiomatizable and true in $(\NN,+)$, and hence it is a subtheory of $\PrA$. But now note that $\mathbf{T}'$ proves that if something was true in the internal structure given by $\iota$, it is true. And since $\mathbf{T}'$ proved any axiom of $\PrA$ in the internal structure given by $\iota$, the theory $\mathbf{T}'$ proves every axiom of $\PrA$. Thus $\mathbf{T}'$ coincides with $\PrA$. But it is known that $\PrA$ is not finitely axiomatizable, contradiction.
	\end{proof}
	
	\section{Visser's Conjecture in Multi-Dimensional Case}
	
	Our goal is to prove \sref{Theorem}{1b}. In order to prove that all multi-dimensional interpretations of $\PrA$ in $(\NN,+)$ are isomorphic to $(\NN,+)$, we use the same argument as in one-dimensional case: an interpretation of a non-standard model would entail an interpretation of a non-scattered order, which is impossible by \sref{Theorem}{ordering}.
	
	However, in order to show that the isomorphism is definable, we first need to develop theory of cardinality functions for the definable families of finite sets.
	
	
	\begin{definition}
	   Let $J\subseteq \ZZ^n$ be a fundamental lattice generated by vectors $\overline{p}_1,\ldots,\overline{p}_m$ from $\overline{c}$.  We say that $f\colon J\to \NN$ is \emph{polynomial} if there is a polynomial with rational coefficients $P_f(x_1,\ldots, x_m)$ such that $f(\overline{c}+\overline{p}_1x_1+\ldots+\overline{p}_mx_m)=P_f(x_1,\ldots, x_m)$, for all $x_1,\ldots,x_m\in\NN$.  
	\end{definition}
	
	We note that if $f$ is a polynomial function on $J$, then the polynomial $P_f$ is uniquely determined.
	
	\begin{definition}
		Let $A\subseteq \ZZ^n$ be a definable set. We call a function $f\colon A\ra\NN$ \emph{piecewise polynomial} if there is a decomposition of $A$ into finitely many fundamental lattices $J_1,\ldots,J_k$ such that the restriction of $f$ on each $J_i$ is a polynomial. The degree $\deg(f)$ is the maximum of the degrees of the restrictions $f\upharpoonright J_i$.
	\end{definition}
	
	We note that our definition of the degree is independent of the choice of the decomposition $J_1,\ldots,J_k$. Indeed, for a piecewise polynomial function $f\colon A\to \NN$ consider the function $h_f\colon \NN\to\NN$ that maps $x\in \NN$ to $\max\{f(\overline{a})\mid \overline{a}\in A \text{ and }|\overline{a}|_{\infty}\le x\}$. Here as usual $|(a_1,\ldots,a_n)|=\max(|a_1|,\ldots,|a_n|)$. Observe that if $f$ has degree $m$ (according to a particular decomposition) then $h_f$ has the asymptotic growth rate of $m$-th degree polynomial. Thus the degree is independent of the choice of decomposition.
	
	By the same argument as above we get the following 
	\begin{lemma} \label{degree_from_bound}Suppose  piecewise polynomial functions $f,g\colon A\to \NN$ are such that $g(\overline{x})\le f(\overline{x})$, for any $\overline{x}$. Then $\deg(g)\le\deg(f)$\end{lemma}
	
	The following theorem is a slight modification of the theorem by G.R.~Blakley~\cite{blakley}.
	
	\begin{theorem}\label{bash}
	    Suppose $M$ is a $m\times n$ matrix of integer numbers. Let the function $\varphi_M\colon\mathbb{Z}^m\ra\NN\cup \{\aleph_0\}$ be defined as follows:
		
		\begin{center}
			$\varphi_M(\overline{u})\eqdef|\{\overline{a}\in\NN^n\mid M\overline{a}=\overline{u}\}|.$
		\end{center}
		Additionally suppose that the values of $\varphi_M$ are always finite. Then $\varphi_M$ is a piecewise polynomial function of the degree $\le n-\rk(M)$.
	\end{theorem}
	
	\begin{proof}In \cite{blakley} it had been proved that $\varphi_M$ is a piecewise polynomial function. Further we prove that $\deg(\varphi_M)\le n-\rk(M)$. Our goal will be to find a polynomial $P(x)$ of the degree $\le n-\rk(M)$ such that $\varphi_M(\overline{u})\le P(|u|_{\infty})$. After this we could derive that $\deg(\varphi_M)\le n-\rk(M)$ by \sref{Lemma}{degree_from_bound}.

          Note that each value $\varphi_M(\overline{u})$ is the number of  natural points (we call $\overline{a}=(a_1,\ldots,a_m)$ natural if $a_1,\ldots,a_m\in\mathbb{N}$) in the hyperplane $H_u=\{\overline{a}\in \RR\mid M\overline{a}=\overline{u}\}$. We are going to find a linear in $|\overline{u}|_{\infty}$ bound on $|\overline{a}|_{\infty}$, for natural points $\overline{a}\in H_{\overline{u}}$. 
          
          Since $\varphi_M(\overline{u})$ is always finite, there could be no non-zero $\overline{a}\in \NN^n$ such that $M\overline{a}=0$. Hence there are no non-zero $\overline{a}\in (\QQ^+)^n$ such that $M\overline{a}=0$. Furthermore, since $M$ was a matrix with integer coefficients, there are no non-zero $\overline{a}\in (\RR^+)^n$ such that $M\overline{a}=0$. Thus there exists a rational $\varepsilon>0$ such that for any $\overline{a}\in (\RR^+)^n$ with $|\overline{a}|_{\infty}=1$ we have $|M\overline{a}|_{\infty}\ge \varepsilon$. Thus for any point $\overline{a}\in H_{\overline{u}}\cap (\RR^+)^n$ we have $\overline{a}\le \frac{|\overline{u}|_\infty}{\varepsilon}$.

   Henceforth all natural points of $H_{\overline{u}}$ are contained in the hypercube $[0,\frac{|\overline{u}|_\infty}{\varepsilon}]^n$. It is easy to see that the intersection of a $k$-dimensional plane with a cube $[0,b]^n$ always contains at most $((b+1)n)^k$ natural points. Given that the planes $H_u$ are $n-\rk(M)$-dimensional, we see that $\varphi_M(\overline{u})\le ((\frac{|\overline{u}|_{\infty}}{\varepsilon}+1)n)^{n-\rk(M)}$. We put $P(x)=((\frac{x}{\varepsilon}+1)n)^{n-\rk(M)}$ and finish the proof.

	\end{proof}
	
	
	\begin{corollary}\label{card_funct}
		For any definable family of finite sets $\langle A_p\subseteq \NN^n\mid p\in P\rangle$, the function $p\mapsto |A_p|$ is piecewise polynomial of the degree $\le n$.
	\end{corollary}
	\begin{proof}
	    Let $A=\bigcup \limits_{p\in P} \{p\frown a\mid a\in A_p\}\subseteq \NN^{m+n}$. We have a decomposition of $A$ into a disjoint union of fundamental lattices $J_1,\ldots,J_n$. A sum of piecewise polynomial functions of degree $\le n$ is piecewise polynomial of the degree $\le n$.  Hence, it is enough to show that for all $J_i$ the function $f_i\colon p\mapsto |J_i\upharpoonright p|$  is a piecewise polynomial function on $P$.
	    
	    Suppose $J_i$ is generated by vectors $v_1,\ldots,v_k$ from $c$. Let $v_1',\ldots,v_k',c'$ be the vectors consisting of first $m$ coordinates of $v_1,\ldots,v_k,c$, respectively.  Let $M$ be the $m\times k$-dimensional matrix corresponding to the function that maps $(x_1,\ldots,x_k)$ to  $v_1'x_1+\ldots v_k'x_k$. It is clear that $\rk(M)\ge k-n$. Now we see that $|J_i\upharpoonright p|=\varphi_M(p-c')$ and thus $f_i$ is piecewise polynomial of the degree $\le n$.
	\end{proof}
	
	\begin{lemma}\label{monotone}
	  Each monotone piecewise polynomial function $f\colon \NN\to\NN$ of the degree $n+1$ is of the form $Cx^{n+1}+g(x)$, where $C>0$ is rational and $g\colon \NN\to\NN$, is piecewise polynomial of the degree $n$.
	\end{lemma}
	\begin{proof}
	    Since $f$ is piecewise polynomial, there is a splitting of $\NN$ into infinite arithmetical progressions and one-element sets $A_1,\ldots,A_n$ such that on each of them $f$ is given by a polynomial $P_1,\ldots,P_n$. From  monotonicity of $f$, it is easy to see that for all infinite $A_i$ the corresponding $P_i$ should have the same highest degree term $Cx^{n+1}$. This determines $g$. On infinite $A_i$, we see that $g(x)=P_i(x)-Cx^{n+1}$ (which is $n$-th degree polynomial). Thus, $g$ is piecewise polynomial of the degree $n$.\end{proof}
	\begin{corollary}\label{difference}
	  Suppose $f\colon \NN\to\NN$ is a monotone piecewise polynomial function of  the degree $n+1$. Then $f(x+1)-f(x)$ is piecewise polynomial of the degree $\le n$.
	\end{corollary}	
	\repeattheorem{1b}
	\begin{proof}
	   As in the proof of \sref{Theorem}{1a} we may assume that the interpretation  of $\mathfrak{A}$ has absolute equality. And we show that $\mathfrak{A}\simeq(\mathbb{N},+)$ by the same method. So further we just prove that the isomorphism is definable.
	
		For $i\in \NN$, let $S_i$ be the maximal initial fragment of $\mathfrak{A}$ such that $|a|_\infty\le i$, for all $a\in S_i$. Clearly, $\langle S_i\mid i\in\NN\rangle$ is a definable family of finite sets. Let $h\colon \NN\to\NN$ be the function $x\mapsto |S_x|$. From \sref{Corollary}{card_funct}, it follows that the function $h$ is piecewise polynomial.   
		
                Clearly, the degree of $h$ is non-zero. First assume that $h$ has the degree $1$. In this case, since $h$ is monotone, from \sref{Corollary}{difference} it follows that $h(x+1)-h(x)$ is piecewise polynomial of the degree $0$ and hence bounded by some constant $C$. Thus, or any $i$ we have $|S_{i+1}\setminus S_i|\le C$. As we will see below this allows us to create a first-order definition of the required isomorphism $f\colon \mathfrak{A}\to (\NN,+)$.

   If $x\in S_0$ we define $f(x)$ by separately considering the cases  $x=a$, for all individual $a\in S_0$. Further we define $f(x)$ for $x\in \mathfrak{A}\setminus S_0$. We find the unique $z$ such that $x\in S_{z+1}\setminus S_{z}$. Let $$U_{x,z}=\{w\in S_{z+1}\setminus S_z\mid w<^{\mathfrak{A}}x\}.$$ Externally we know that $f(x)=h(z)+|U_{x,z}|$. Since $h$ is piecewise linear, by \sref{Theorem}{jetz} it is definable. We know that $0\le |U_{x,z}|<C$, which allows us to define the value $f(x)$ by separately considering the cases for all possible values of $|U_{x,z}|$. More formally this description corresponds to the following definition of the predicate $f(x)=y$:
                $$\bigwedge\limits_{a\in S_0}\big(x=a\to y=f(a)\big)\land \bigwedge\limits_{0\le s<C} \forall z\big(x\in S_{z+1}\setminus S_z\land |U_{x,z}|=s\to y=h(z)+s\big),$$
                where for each $s<C$ the property $|U_{x,z}|=s$ is defined by the formula
          $$(\exists! w_1,\ldots,w_s)\big((w_0,\ldots,w_{s-1}\in S_{z+1}\setminus S_z)\land \bigwedge\limits_{i<j< s}w_i\ne w_j\land \bigwedge\limits_{i<s}w_i<^{\mathfrak{A}}x\big).$$      
		
          Now assume that $h$ has the degree $k\ge 2$. Our goal will be to show that this is in fact impossible. For this we consider the following definable function $g\colon \NN\to\NN$:
          $$g(x)=\min\{ y\mid (\forall z\in S_x) (z+^{\mathfrak{A}}z+^{\mathfrak{A}}+1^{\mathfrak{A}}\in S_y)\}$$
          In other words $g(x)$ is the least $y$ such that the initial fragment $S_y$ is at least two times larger than $S_x$.  Thus we have $$h(g(x)-1)<2h(x)\le h(g(x)).$$

    Since both $h$ and $g$ are  monotone, by \sref{Lemma}{monotone} we have rational $C_1,C_2>0$ such that $h(x)=C_1x^k(1+o(1))$ and $g(x)=C_2x(1+o(1))$. Therefore $h(g(x)-1)=C_1C_2^kx^k(1+o(1))$ and $h(g(x))=C_1C_2^kx^k(1+o(1))$. Hence $2h(x)=C_1C_2^kx^k(1+o(1))$. At the same time $2h(x)=2C_1x^k(1+o(1))$. Thus $2=C_2^k$ and $C_2=\sqrt[k]{2}$. Contradiction with the fact that $C_2$ is rational.\end{proof}		
			In the same manner as \sref{Theorem}{visser_1} (but using \sref{Theorem}{1b} instead of \sref{Theorem}{1a}) we prove
			
	\begin{theorem}
	    Theory $\PrA$ is not interpretable in any of its finitely axiomatizable subtheories.
	\end{theorem}

	\section*{Acknowledgments}
	
	The authors thank Lev~Beklemishev for suggesting to study Visser's conjecture, a number of fruitful discussions of the subject, and his useful comments.
	
	Work of Fedor Pakhomov is supported by grant 19-05497S of GA \v{C}R.
	
	Work of Alexander Zapryagaev was prepared within the framework of the Academic Fund Program at the National Research University Higher School of Economics (HSE) in 2020 (grant No. 19-04-050) and by the Russian Academic Excellence Project ``5-100".

\end{document}